\documentclass[a4paper, USenglish, cleveref, autoref, thm-restate]{lipics-v2021}

\hideLIPIcs  

\graphicspath{{./figure/}}

\title{Hardness of Finding Combinatorial Shortest Paths on Graph Associahedra} 

\author{Takehiro Ito}{Graduate School of Information Sciences, Tohoku University, Japan}{takehiro@tohoku.ac.jp}{https://orcid.org/0000-0002-9912-6898}{JSPS KAKENHI Grant Numbers JP18H04091, JP19K11814, JP20H05793}
\author{Naonori Kakimura}{Faculty of Science and Technology, Keio University, Japan}{kakimura@math.keio.ac.jp}{https://orcid.org/0000-0002-3918-3479}{JSPS KAKENHI Grant Numbers 	JP20H05795, JP21H03397, JP22H05001}
\author{Naoyuki Kamiyama}{Institute of Mathematics for Industry, Kyushu University, Japan}{kamiyama@imi.kyushu-u.ac.jp}{https://orcid.org/0000-0002-7712-2730}{JSPS KAKENHI Grant Number JP20H05795}
\author{Yusuke Kobayashi}{Research Institute for Mathematical Sciences, Kyoto University, Japan}{yusuke@kurims.kyoto-u.ac.jp}{https://orcid.org/0000-0001-9478-7307}{JSPS KAKENHI Grant Numbers JP20K11692, JP20H05795, JP22H05001}
\author{Shun-ichi Maezawa}{Department of Mathematics, Tokyo University of Science, Japan}{maezawa.mw@gmail.com}{https://orcid.org/0000-0003-1607-8665}{JSPS KAKENHI Grant Numbers JP20H05795, JP22K13956}
\author{Yuta Nozaki}{Faculty of Environment and Information Sciences, Yokohama National University, Japan \and SKCM$^2$, Hiroshima University, Japan}{nozaki-yuta-vn@ynu.ac.jp}{https://orcid.org/0000-0003-3223-0153}{JSPS KAKENHI Grant Numbers JP20H05795, JP20K14317, JP23K12974}
\author{Yoshio Okamoto}{Graduate School of Informatics and Engineering, The University of Electro-Communications, Japan}{okamotoy@uec.ac.jp}{https://orcid.org/0000-0002-9826-7074}{JSPS KAKENHI Grant Numbers JP20H05795, JP20K11670, JP23K10982}

\authorrunning{T. Ito et al.} 

\Copyright{Takehiro Ito, Naonori Kakimura, Naoyuki Kamiyama, Yusuke Kobayashi, Shun-ichi Maezawa, Yuta Nozaki, and Yoshio Okamoto} 



\category{} 

\acknowledgements{This work was motivated by the talks of Vincent Pilaud and Jean Cardinal at Workshop ``Polytope Diameter and Related Topics,'' held online on September 2, 2022. We thank them for inspiration. We are also grateful to Yuni Iwamasa and Kenta Ozeki for the discussion and to the anonymous reviewers for their helpful comments.}

\nolinenumbers 

\newcommand{\ini}{\mathrm{ini}} 
\newcommand{\tar}{\mathrm{tar}} 

\begin{document}

\maketitle

\begin{abstract}
We prove that the computation of a combinatorial shortest path between two vertices of a graph associahedron, introduced by Carr and Devadoss, is NP-hard.
This resolves an open problem raised by Cardinal.
A graph associahedron is a generalization of the well-known associahedron.
The associahedron is obtained as the graph associahedron of a path.
It is a tantalizing and important open problem in theoretical computer science whether the computation of a combinatorial shortest path between two vertices of the associahedron can be done in polynomial time, which is identical to the computation of the flip distance between two triangulations of a convex polygon, and the rotation distance between two rooted binary trees.
Our result shows that a certain generalized approach to tackling this open problem is not promising.
As a corollary of our theorem, we prove that the computation of a combinatorial shortest path between two vertices of a polymatroid base polytope cannot be done in polynomial time unless $\mathrm{P} = \mathrm{NP}$.
Since a combinatorial shortest path on the matroid base polytope can be computed in polynomial time, our result reveals an unexpected contrast between matroids and polymatroids.
\end{abstract}

\section{Introduction}
Graph associahedra were introduced by Carr and Devadoss \cite{CARR20062155}.
These polytopes generalize associahedra.
In an associahedron, each vertex corresponds to a binary tree over a set of $n$ elements, and each edge corresponds to a rotation operation between two binary trees.
For the historical account of associahedra, see the introduction of the paper by Ceballos, Santos, and Ziegler~\cite{DBLP:journals/combinatorica/CeballosSZ15}.

A binary tree can be obtained from a labeled path.
Let $V=\{1,2,\dots,n\}$ be the set of vertices of the path, and $E=\{\{i,i+1\} \mid 1\leq i\leq n-1\}$ be the set of edges of the path.
To construct a labeled binary tree, we choose an arbitrary vertex from the path.
Let it be $i \in V$.
Then, the removal of $i$ from the path results in at most two connected components: the left subpath and the right subpath, which may be empty.
Then, in the corresponding binary tree, we set $i$ as a root, and append recursively a binary tree of the left subpath as a left subtree and a binary tree of the right subpath as a right subtree.
Note that in this construction, each node of the binary tree is labeled by a vertex of the path.

In the construction of graph associahedra, we follow the same idea.
Since we are only interested in the graph structure of graph associahedra in this work, we only describe their vertices and edges.
To define a graph associahedron, we first fix a connected undirected graph $G = (V, E)$.\footnote{A graph associahedron can also be defined for disconnected graphs, but in this paper, we concentrate on connected graphs.}
Then, in the $G$-associahedron, the vertices correspond to the so-called elimination trees of $G$, and the edges correspond to swap operations between two elimination trees.
The following description follows that of
Cardinal, Merino, and M\"utze~\cite{DBLP:conf/soda/CardinalMM22}.

An \emph{elimination tree} of a connected undirected graph $G = (V, E)$ is a rooted tree defined as follows.
It has $V$ as the vertex set and is composed of a root $v \in V$ that has as children elimination trees for each connected component of $G-v$ (\figurename~\ref{fig:elimtree_example1}).
A swap from an elimination tree $T$ to another elimination tree $T'$ of $G$ is defined as follows.
Let $v$ be a non-root vertex of $T$, and let $u$ be the parent of $v$ in $T$.
Denote by $H$ the subgraph of $G$ induced by the subtree rooted at $v$ in $T$.
Then, the swap of $u$ with $v$ transforms $T$ to $T'$ as follows.
(1) The tree $T'$ has $v$ as the parent of $u$, and $T'$ has $v$ as a child of the parent of $u$ in $T$.
(2) The subtrees rooted at $u$ in $T$ remain subtrees rooted at $u$ in $T'$. 
(3) A subtree $S$ rooted at $v$ in $T$ remains a subtree rooted at $v$ in $T'$, unless the vertices of $S$ belong to the same connected component of $H-v$ as $u$, in which case $S$ becomes a subtree rooted at $u$ in $T'$.
The $G$-associahedron for a claw $G$ is given in \figurename~\ref{fig:stellohedron1}.
Note that a swap operation is reversible.

\begin{figure}[t]
\centering
\includegraphics{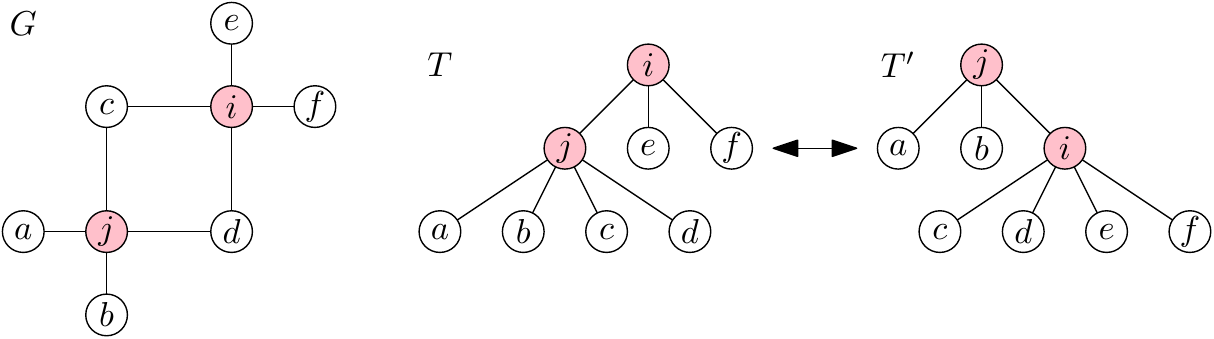}
\caption{An example of elimination trees.
Two trees $T$ and $T'$ are elimination trees of the graph $G$.
The tree $T'$ is obtained from $T$ by the swap of $i$ with $j$.
The example is borrowed from Cardinal, Merino, and M\"utze~\cite{DBLP:conf/soda/CardinalMM22}.}
\label{fig:elimtree_example1}
\end{figure}

\begin{figure}[t]
\centering
\includegraphics{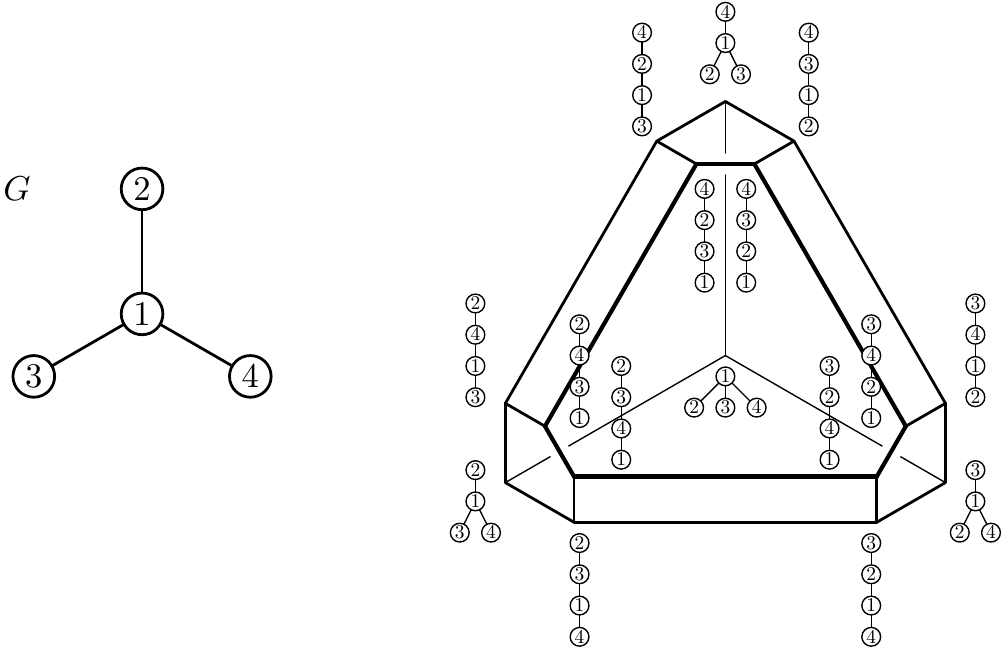}
\caption{An example of a graph associahedron. Each vertex of the polytope corresponds to an elimination tree of the graph $G$.}
\label{fig:stellohedron1}
\end{figure}

In this paper, among graph properties of graph associahedra, we concentrate on the computation of a combinatorial shortest path (i.e., the graph-theoretic distance) between two vertices of the polytope, which we call the combinatorial shortest path problem on graph associahedra.
In this problem, we are given a graph $G$ and two elimination trees $T, T'$ of $G$, and want to compute the shortest length of a graph-theoretic path from $T$ to $T'$ on the $G$-associahedron.
In the literature, we only find the studies in the case where $G$ is a complete graph or (a generalization of) a star.
When $G$ is a complete graph, the $G$-associahedron is called a permutahedron, and each of its vertices corresponds to a permutation.
Since an edge corresponds to an adjacent transposition, the graph-theoretic distance between two vertices is equal to the number of inversions between the corresponding permutations.
This can be computed in polynomial time.
When $G$ is a star, the $G$-associahedron is called stellohedron~\cite{MR2520477}.
Recently, Cardinal, Pournin, and Valencia-Pabon~\cite{https://doi.org/10.48550/arxiv.2211.07984} gave a polynomial-time algorithm to compute a combinatorial shortest path on stellohedra, and they generalize the algorithm when $G$ is a complete split graph (i.e., a graph obtained from a star by replacing the center vertex with a clique).

On the other hand, it is a tantalizing open problem whether a combinatorial shortest path can be computed in polynomial time when $G$ is a path.
In this case, the graph-theoretic distance corresponds to the rotation distance between two binary trees.
By Catalan correspondences, this is equivalent to the flip distance between two triangulations of a convex polygon.
A possible strategy to resolve this open problem is to generalize the problem and solve the generalized problem.
In our case, a generalization is achieved by considering graph associahedra for general graphs.

Our main result states that the combinatorial shortest path problem on $G$-associahedra is $\mathrm{NP}$-hard when $G$ is also given as part of the input.
This implies that the strategy mentioned above is bound to fail.

First, we formally state the problem \textsc{Combinatorial Shortest Path on Graph Associahedra} as follows.
\begin{center}    
\begin{tabular}{p{0.9\textwidth}}
\hline
\textsc{Combinatorial Shortest Path on Graph Associahedra}\\
\hline
Input: A graph $G$ and two elimination trees $T_\ini, T_\tar$ of $G$\\
Output: The distance between $T_\ini$ and $T_\tar$ on the graph of the $G$-associahedron\\
\hline
\end{tabular}
\end{center}

Our first theorem states the $\mathrm{NP}$-hardness of \textsc{Combinatorial Shortest Path on Graph Associahedra}.
This solves an open problem raised by Cardinal (see \cite[Section~4.2]{BuchinLMS22}).
\begin{theorem}
\label{thm:hardness}
\textsc{Combinatorial Shortest Path on Graph Associahedra} is $\mathrm{NP}$-hard. 
\end{theorem}

Theorem \ref{thm:hardness} yields the following corollary, which is related to polymatroids introduced by Edmonds~\cite{E70}.
A pair $(U,\rho)$ of a finite set $U$ and 
a function $\rho \colon 2^U \to \mathbb{R}$ is called 
a \emph{polymatroid} if $\rho$ satisfies 
the following conditions:
(P1) $\rho(\emptyset) = 0$;
(P2) if $X \subseteq Y \subseteq U$, then $\rho(X) \le \rho(Y)$;
(P3) if $X, Y \subseteq U$, then $\rho(X \cup Y) + \rho(X \cap Y) \le \rho(X) + \rho(Y)$.
The function $\rho$ is called the \emph{rank function} of the polymatroid $(U, \rho)$.

For a polymatroid $(U,\rho)$, 
we define the \emph{base polytope} of $(U, \rho)$ as
\[
\mathbf{B}(\rho) := \{x \in \mathbb{R}^U \mid x(X) \le \rho(X) \, (\forall\, X \subseteq U), \, 
x(U) = \rho(U)\},
\]
where we define $x(X) := \sum_{u \in X}x(u)$ for each 
subset $X \subseteq U$. 
Then, $\mathbf{B}(\rho)$ is a polytope because 
$0 \le \rho(U) - \rho(U\setminus \{u\}) = x(U) - \rho(U\setminus \{u\}) \le
x(u) \le \rho(\{u\})$ 
for every element $e \in E$. 

A polymatroid is seen as a polyhedral generalization of a matroid.
For example, the greedy algorithm for matroids can be treated as an algorithm to maximize a linear function over the base polytope of a matroid,\footnote{%
This can further be seen as a generalization of Kruskal's algorithm for the minimum spanning tree problem.}
and the algorithm is readily generalized to the base polytope of a polymatroid.
A lot of combinatorial properties of the base polytopes of matroids also hold for the base polytopes of polymatroids.
Since it is known that a combinatorial shortest path on the base polytope of a matroid can be computed in polynomial time \cite{DBLP:journals/tcs/ItoDHPSUU11}, we are interested in generalizing this result to polymatroids, which leads to the following problem definition.

\begin{center}    
\begin{tabular}{p{0.9\textwidth}}
\hline
\textsc{Combinatorial Shortest Path on Polymatroids}\\
\hline
Input: An oracle access to a polymatroid $(U, \rho)$ and two extreme points $x_\ini, x_\tar$ of the base polytope $\mathbf{B}(\rho)$\\
Output: The distance between $x_\ini$ and $x_\tar$ on $\mathbf{B}(\rho)$\\
\hline
\end{tabular}
\end{center}

We note that a polymatroid $(U, \rho)$ is given as an oracle access that returns the value $\rho(X)$ for any set $X\subseteq U$.
The running time of an algorithm is also measured in terms of the number of oracle calls.
This is a standard assumption when we deal with polymatroids~\cite{F05} since if we would input the function $\rho$ as a table of the values $\rho(X)$ for all $X \subseteq U$, then it would already take at least $2^{|U|}$ space.
We also note that the adjacency of two extreme points of the base polytope of a polymatroid can be tested in polynomial time~\cite{DBLP:journals/mp/Topkis84}.

The next theorem states that this problem is hard, which is proved as a corollary of Theorem \ref{thm:hardness}, and reveals an unexpected contrast between matroids and polymatroids.

\begin{theorem}
\label{thm:hardness-polymatroid}
There exists no polynomial-time algorithm for
\textsc{Combinatorial Shortest Path on Polymatroid}
unless $\mathrm{P} = \mathrm{NP}$.
\end{theorem}
Our proof relies on the fact that graph associahedra can be realized as the base polytopes of some polymatroids~\cite{MR2520477}.
However, we need careful treatment since in the reduction we require the rank function of our polymatroid to be evaluated in polynomial time.
To this end, we give an explicit inequality description of the realization of a graph associahedron due to Devadoss~\cite{DBLP:journals/dm/Devadoss09},\footnote{%
We note here that the original definition of a graph associahedron by Carr and Devadoss~\cite{CARR20062155} does not give explicit vertex coordinates of the polytope. Therefore, we rely on the realization by Devadoss~\cite{DBLP:journals/dm/Devadoss09} who gave the explicit vertex coordinates.
}
which is indeed the base polytope of a polymatroid.

\subsection*{Related Work}

Paths on polytopes have been studied thoroughly.
One of the initial motivations for this research direction is to design and understand path-following algorithms for linear optimization such as simplex methods.
In his chapter of \textit{Handbook of Discrete and Computational Geometry}~\cite{DBLP:reference/cg/2017}, Kalai stated as an open problem ``Find an efficient routing algorithm for convex polytopes.''
Here, a routing algorithm means one that finds a short path from a given initial vertex to a given target vertex.

Paths on graph associahedra have been receiving much attention.
The diameter is perhaps the most frequently studied quantity, which is defined as the maximum distance between two vertices of the polytope.
A famous result by Sleator, Tarjan, and Thurston \cite{STT-jams} states that the diameter of the $(n-1)$-dimensional associahedron (i.e., a graph associahedron of an $n$-vertex path) is at most $2n-6$ when $n\geq 11$ and this bound is tight for all sufficiently large $n$.
This bound is refined by Pournin~\cite{POURNIN201413}, who proved that the diameter of the $(n-1)$-dimensional associahedron is exactly $2n-6$ when $n \geq 11$.

For a general $n$-vertex graph $G$, Manneville and Pilaud~\cite{manneville:hal-01276806} proved that the diameter of $G$-associahedron is at most $\binom{n}{2}$ and at least $\max\{m,2n-20\}$, when 
$m$ is the number of edges of $G$.
The upper bound is attained by the case where $G$ is a complete graph (i.e., the $G$-associahedron is a permutahedron).
When $G$ is an $n$-vertex star (i.e., $K_{1,n-1}$), $n\geq 6$, Manneville and Pilaud~\cite{manneville:hal-01276806} showed that the diameter is $2n-2$.
When $G$ is a cycle (i.e., the polytope is a cyclohedron), 
Pournin \cite{POURNIN-cyclohedra} gave the asymptotically exact diameter.
When $G$ is a tree, Manneville and Pilaud~\cite{manneville:hal-01276806} gave the upper bound $O(n \log n)$ while 
Cardinal, Langerman, and P{\'{e}}rez{-}Lantero \cite{DBLP:journals/combinatorics/CardinalLP18} gave an example of trees for which the diameter is $\Omega(n \log n)$ (such an example is chosen as a complete binary tree).
Cardinal, Pournin, and Valencia-Pabon~\cite{defga-CardinalPVP} proved that the diameter is $\Theta(m)$ for $m$-edge trivially perfect graphs, and gave upper and lower bounds for the diameter in terms of treewidths, pathwidths, and treedepths of graphs.
Berendsohn~\cite{DBLP:conf/swat/Berendsohn22} proved that the diameter is $\Theta(n + m H)$ for a caterpillar with $n$ spine vertices, $m$ leg vertices, and the Shannon entropy $H$ of the so-called leg distribution.

To the authors' knowledge, 
the complexity of computing the diameter of graph associahedra has not been investigated.
When polytopes are not restricted to graph associahedra, a few hardness results have been known.
Frieze and Teng~\cite{DBLP:journals/cc/FriezeT94} proved that computing the diameter of a polytope, given by inequalities, is weakly NP-hard.
Sanit\`{a} \cite{DBLP:conf/focs/Sanita18} proved that computing the diameter of the fractional matching polytope of a given graph is strongly NP-hard.
Kaibel and Pfetsch \cite{DBLP:conf/dagstuhl/KaibelP03} raised an open problem about the complexity of computing the diameter of simple polytopes.

The computation of a combinatorial shortest path on a $G$-associahedron has also been studied.
It is a long-standing open problem whether a combinatorial shortest path in an associahedron (i.e., a $G$-associahedron when $G$ is a path) can be computed in polynomial time.
Polynomial-time algorithms are only known when $G$ is a complete graph (folklore), a star, or a complete split graph (Cardinal, Pournin, and Valencia-Pabon~\cite{https://doi.org/10.48550/arxiv.2211.07984}).
When $G$ is a path, a polynomial-time approximation algorithm of factor two~\cite{DBLP:journals/jgaa/ClearyJ10} and fixed-parameter algorithms when the distance is a parameter~\cite{DBLP:journals/ipl/ClearyJ09,DBLP:journals/dcg/KanjSX17,DBLP:conf/stacs/LiX23,DBLP:journals/ipl/Lucas10} are known.

Since a combinatorial shortest path on an associahedron is equivalent to a shortest flip sequence of triangulations of a convex polygon, the computation of a shortest flip sequence of triangulations has been studied in more general setups.
For triangulations of point sets, the problem is $\mathrm{NP}$-hard~\cite{DBLP:journals/comgeo/LubiwP15} and even $\mathrm{APX}$-hard~\cite{DBLP:journals/comgeo/Pilz14}.
For triangulations of simple polygons, the problem is also $\mathrm{NP}$-hard~\cite{DBLP:journals/dcg/AichholzerMP15}.

Elimination trees have appeared in a lot of branches of mathematics and computer science.
For a good summary, see Cardinal, Merino, and M\"utze~\cite{DBLP:conf/soda/CardinalMM22}.

\subsection*{Technique Overview}

To prove the hardness of the combinatorial shortest path problem on graph associahedra, 
we first consider a ``weighted'' version of the combinatorial shortest path problem on graph associahedra, which is newly introduced in this paper for our reduction.
In this problem, each vertex of a given graph has a positive weight, and the swap of two vertices incurs the weight that is defined as the product of the weights of these two vertices.
The weight of a swap sequence is defined as the sum of weights of swaps in the sequence.
As our intermediate theorem, we prove that this weighted version is strongly $\mathrm{NP}$-hard.

To this end, we reduce the $\mathrm{NP}$-hard problem of finding a balanced minimum $s$-$t$ cut in a graph~\cite{DBLP:conf/stoc/FeigeM06} to the weighted version of the combinatorial shortest path problem on graph associahedra.
In the balanced minimum $s$-$t$ cut problem, we want to determine whether there exists a minimum $s$-$t$ cut of a given graph $G$ that is a bisection of the vertex set.
In the reduction, we construct a vertex-weighted graph $H$ from $G$ and two elimination trees $T_\ini$, $T_\tar$ of $H$.
The weighted graph $H$ is constructed by replacing $s$ and $t$ by large cliques,
subdividing each edge, and duplicating each vertex; the weights are assigned so that the subdivision vertices receive small weights, and duplicated vertices and vertices in large cliques receive large weights.
Elimination trees $T_\ini$ and $T_\tar$ are constructed so that swapping two vertices of large weights is forbidden and identifying a few vertices of small weights (that corresponds to a balanced minimum $s$-$t$ cut of $G$) gives a shortest path.

In the second step, we reduce the weighted version to the original unweighted version of the problem.
To this end, a vertex of weight $w$ is replicated by a clique of size $w$.
We want to make sure that a swap of the vertices $u, v$ of weights $w(u), w(v)$, respectively in the weighted instance is mapped to consecutive $w(u)\cdot w(v)$ swaps of the vertices of cliques that correspond to $u$ and $v$ in the constructed unweighted instance.
This is proved via the useful operation of projections combined with the averaging argument.

\section{Preliminaries}
For a positive integer $k$, 
let $[k]$ denote $\{1,2,\dots,k\}$.

For a graph $G=(V, E)$ and two elimination trees $T_{\ini}$ and $T_{\tar}$ of $G$, 
we say that a sequence $\mathbf{T} = \langle T_0, T_1, \dots , T_\ell \rangle$ of elimination trees of $G$ is 
a \emph{reconfiguration sequence from $T_{\ini}$ to $T_{\tar}$} if $T_0=T_{\ini}$, $T_{\ell} = T_{\tar}$, and $T_i$ is obtained from $T_{i-1}$ by applying a single swap operation for $i \in [\ell]$. 
We sometimes regard $\mathbf{T}$ as a sequence of swap operations if no confusion may arise. 
The length of $\mathbf{T}$ is defined as the number $\ell$ of swaps in $\mathbf{T}$, which we denote $\mathrm{length}(\mathbf{T})$. 
Then, \textsc{Combinatorial Shortest Path on Graph Associahedra}
is the problem of finding a reconfiguration sequence $\mathbf{T}$ from $T_{\ini}$ to $T_{\tar}$ that minimizes $\mathrm{length} (\mathbf{T})$. 

When $u \in V$ is a child of $v \in V$ in an elimination tree $T$, an operation swapping $u$ and $v$ is represented by $\mathbf{swap} (u, v)$. 
Note that we distinguish $\mathbf{swap} (u, v)$ and $\mathbf{swap} (v, u)$.
For an elimination tree
$T$ and for a vertex $v \in V(T)$, let $\mathbf{anc}_T(v)$ (resp.~$\mathbf{des}_T(v)$) denote the set of all ancestors (descendants) of $v$ in $T$. 
Note that $u \in \mathbf{anc}_T(v)$ if and only if $v \in \mathbf{des}_T(u)$. 
Note also that neither $\mathbf{anc}_T(v)$ nor $\mathbf{des}_T(v)$ contains $v$. 
We say that distinct vertices $u$ and $v$ are \emph{comparable} in $T$ if $u \in \mathbf{anc}_T(v)$ or $v \in \mathbf{anc}_T(u)$. 
Otherwise, they are called \emph{incomparable} in $T$. 
A linear ordering $\prec$ on $V$ defines an elimination tree $T$ uniquely so that 
$u \in \mathbf{anc}_T(v)$ implies $u \prec v$. 

Let $G=(V, E)$ be an undirected graph. 
For $X \subseteq V$, let $\delta_G(X)$ denote the set of edges between $X$ and $V \setminus X$. 
For $s, t \in V$, we say that $X \subseteq V$ is an \emph{$s$-$t$ cut} if $s \in X$ and $t \not\in X$. An edge set $F \subseteq E$ is called an \emph{$s$-$t$ cut set} if $F = \delta_G(X)$ for some $s$-$t$ cut $X \subseteq V$. A \emph{minimum $s$-$t$ cut} is an $s$-$t$ cut $X$ minimizing $|\delta_G(X)|$.
For $X \subseteq V$, let $G[X]$ denote the subgraph induced by $X$ and 
let $E[X]$ denote its edge set.

\section{Hardness of the Weighted Problem}
\label{sec:hardnessweighted}

We consider a weighted variant of \textsc{Combinatorial Shortest Path on Graph Associahedra}, 
which we call \textsc{Weighted Combinatorial Shortest Path on Graph Associahedra}. 
In the problem, we are given a graph $G=(V, E)$, two elimination trees $T_{\ini}$ and $T_{\tar}$, 
and a weight function $w \colon V \to \mathbb{Z}_{>0}$. 
For $u, v \in V$, the \emph{weight} of $\mathbf{swap} (u, v)$ is defined as $w(u) \cdot w(v)$. 
This value is sometimes denoted by $w (\mathbf{swap} (u, v))$. 
The \emph{weighted length} (or simply the \emph{weight}) of a reconfiguration sequence $\mathbf{T}$ is defined as 
the total weight of swaps in $\mathbf{T}$, which we denote by $\mathrm{length}_w (\mathbf{T})$.
The objective of  \textsc{Weighted Combinatorial Shortest Path on Graph Associahedra}
is to find a reconfiguration sequence $\mathbf{T}$ from $T_{\ini}$ to $T_{\tar}$ that minimizes $\mathrm{length}_w (\mathbf{T})$. 

In this section, we show that the weighted variant is strongly $\mathrm{NP}$-hard. 

\begin{theorem}
\label{thm:hardnessweighted}
\textsc{Weighted Combinatorial Shortest Path on Graph Associahedra} is strongly $\mathrm{NP}$-hard, that is, 
it is $\mathrm{NP}$-hard even when the input size is $\sum_{v \in V} w(v)$. 
\end{theorem}

\subsection{Reduction}

To show Theorem~\ref{thm:hardnessweighted}, we reduce \textsc{Balanced Minimum $s$-$t$ Cut} to \textsc{Weighted Combinatorial Shortest Path on Graph Associahedra}. 
In \textsc{Balanced Minimum $s$-$t$ Cut}, the input consists of a connected graph $G=(V, E)$ with $s, t \in V$, 
and the objective is to determine whether $G$ contains a minimum $s$-$t$ cut $X$ with $|X| = |V \setminus X|$.
Without loss of generality, we may assume that $|V|$ is even. 
Let $V = \{s, t, v_1, v_2, \dots , v_{2n} \}$ and $E = \{e_1, \dots , e_m\}$, where $|V| = 2n+2$ and $|E| = m$.  
It is known that \textsc{Balanced Minimum $s$-$t$ Cut} is $\mathrm{NP}$-hard~\cite{DBLP:conf/stoc/FeigeM06}.
For an instance of \textsc{Balanced Minimum $s$-$t$ Cut}, we construct an instance of \textsc{Weighted Combinatorial Shortest Path on Graph Associahedra} as follows. 

Let $N$ be a sufficiently large integer (e.g., $N = 10 n^3 m$). 
We first subdivide each edge $e \in E$ by introducing a new vertex $u_e$. 
Then, for each $v \in V$, we introduce a copy $v'$ of $v$. 
We replace $s$ with a clique $\{s_1, \dots , s_{N^3} \}$ of size $N^3$ and replace $t$ with another clique $\{t_1, \dots , t_{N^3} \}$ of size $N^3$. 
Let $H$ be the obtained graph. 
Formally, the graph $H = (V(H), E(H))$ is defined as follows:  
\begin{align*}
V(H) &= (V \setminus \{s, t\}) \cup  \{v' \mid v \in V \} \cup \{u_e \mid e \in E\} \cup  \{s_1, \dots , s_{N^3} \} \cup \{t_1, \dots , t_{N^3} \}, \\
E(H) &= \{ \{v, u_e\} \mid v \in V \setminus \{s, t\},\ e \in \delta_G(v)\} \cup \{ \{v', u_e\} \mid v \in V,\ e \in \delta_G(v)\} \\ 
	 & \quad \cup \{ \{s_i, s_j\} \mid i, j \in [N^3],\ i \neq j \} \cup \{ \{t_i, t_j\} \mid i, j \in [N^3],\ i \neq j \} \\ 
	 & \quad \cup \{ \{s_i, u_e\} \mid i \in [N^3],\ e \in \delta_G(s) \} \cup \{ \{t_i, u_e\} \mid i \in [N^3],\ e \in \delta_G(t) \} .  
\end{align*}
Define $w \colon V(H) \to \mathbb{Z}_{>0}$ as follows: 
\begin{alignat*}{2}
&w(v) =  N  &\qquad&(v \in V \setminus \{s, t\}),  \\
&w(v') =  N^8  &\qquad&(v \in V),  \\
&w(u_e) = 1   &\qquad&(e \in E), \\
&w(s_i) = w(t_i) = N^4  &\qquad&(i \in [N^3]). 
\end{alignat*}
The initial elimination tree $T_{\ini}$ is defined by the following linear ordering:  
\begin{align*}
&v_1 \prec \dots \prec v_{2n}  \prec s_1 \prec t_1 \prec s_2 \prec t_2 \prec \dots \prec s_{N^3} \prec t_{N^3}  \\
& \qquad  \prec u_{e_1} \prec \dots \prec u_{e_m} \prec v'_1 \prec \dots \prec v'_{2n} \prec s' \prec t'.  
\end{align*}
Note that, in $T_{\ini}$,  the vertices
$v_1, \dots , v_{2n}, s_1, t_1, s_2, t_2, \dots , s_{N^3}, t_{N^3}$ are aligned on a path, while the other elements are not necessarily aligned sequentially.  
The target elimination tree $T_{\tar}$ is the elimination tree defined by the following linear ordering:  
\begin{align*}
&v_{2n} \prec \dots \prec v_{1}  \prec t_1 \prec s_1 \prec t_2 \prec s_2 \prec \dots \prec t_{N^3} \prec s_{N^3}  \\
&\qquad  \prec u_{e_1} \prec \dots \prec u_{e_m} \prec v'_1 \prec \dots \prec v'_{2n} \prec s' \prec t'.  
\end{align*}
We consider an instance $(H, w, T_{\ini}, T_{\tar})$ of \textsc{Weighted Combinatorial Shortest Path on Graph Associahedra}. 
In this instance, we reverse the ordering of the first $2n$ elements and reverse the ordering of $s_i$ and $t_i$ for each $i$. 
See \figurename~\ref{fig:reduction1} for an illustration.

\begin{figure}[t]
\centering
\includegraphics{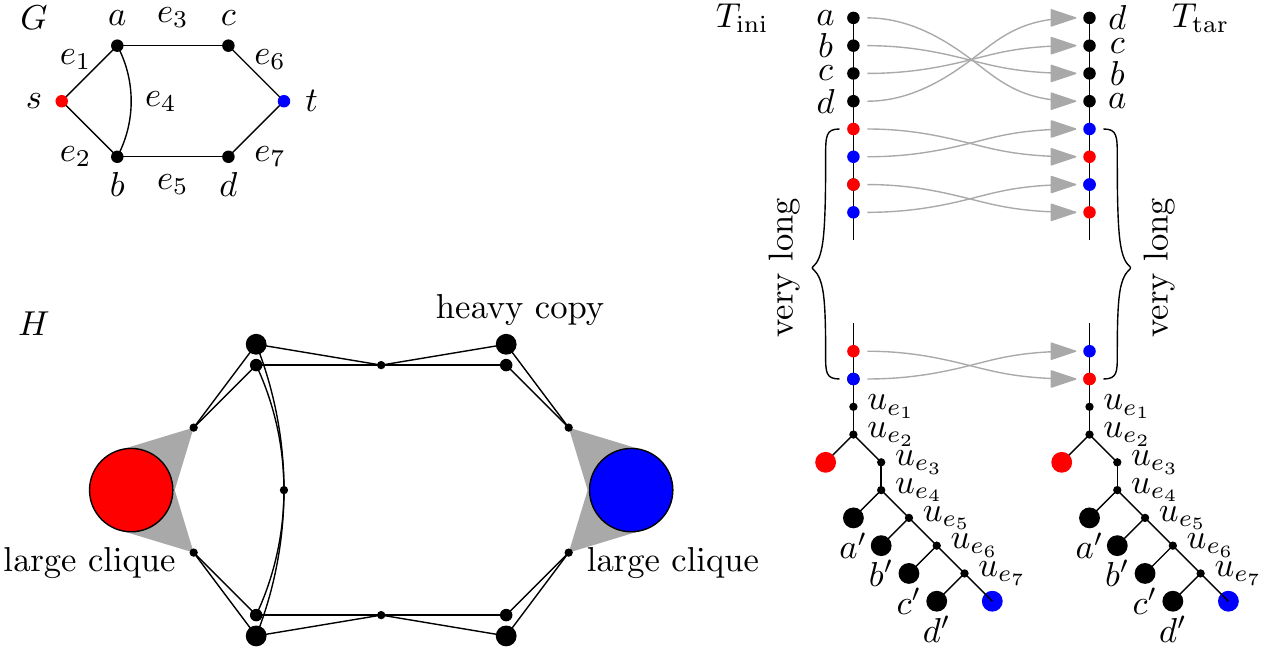}
\caption{Reduction for Theorem \ref{thm:hardnessweighted}.}
\label{fig:reduction1}
\end{figure}

To prove Theorem~\ref{thm:hardness}, it suffices to show the following proposition. 

\begin{proposition}
\label{prop:01}
Let $\lambda$ be the minimum size of an $s$-$t$ cut set in $G$. 
There is a reconfiguration sequence from  $T_{\ini}$ to $T_{\tar}$ of weight less than $4 \lambda N^7 + (n^2 - n +1)N^2$
if and only if 
$G$ has a minimum $s$-$t$ cut $X$ with $|X| = |V \setminus X|$. 
\end{proposition}

\subsection{Proof of Proposition~\ref{prop:01}}

\subsubsection*{Sufficiency (``if'' part)}

Suppose that $G$ has a minimum $s$-$t$ cut $X$ with $|X| = |V \setminus X| = n+1$. 
Let $U = \{u_e \mid e \in \delta_G(X)\}$. Note that $|U| = |\delta_G(X)| = \lambda$.  
Starting from $T_{\ini}$, we swap an element in $U$ and its parent repeatedly so that 
we obtain an elimination tree $T_1$ in which each element in $U$ is an ancestor of $V(H) \setminus U$. 
See \figurename~\ref{fig:reduction2}.
The total weight of swaps from $T_\ini$ to $T_1$ is at most $|U| ( 2n N + 2 N^7 + m)$. 
Since $G - \delta_G(X)$ consists of two connected components, so does $H - U$. 
Thus, $T_1 - U$ consists of two elimination trees $T_s$ and $T_t$ such that
$T_s$ contains $(X \setminus \{s\})  \cup \{s_1, \dots , s_{N^3}\} \cup \{ u_e \mid e \in E[X]\} \cup \{v' \mid v \in X\}$ 
and $T_t$ contains $((V \setminus X) \setminus \{t\}) \cup \{t_1, \dots , t_{N^3}\} \cup \{ u_e \mid e \in E[V \setminus X]\} \cup \{v' \mid v \in V \setminus X \}$.

\begin{figure}[t]
\centering
\includegraphics{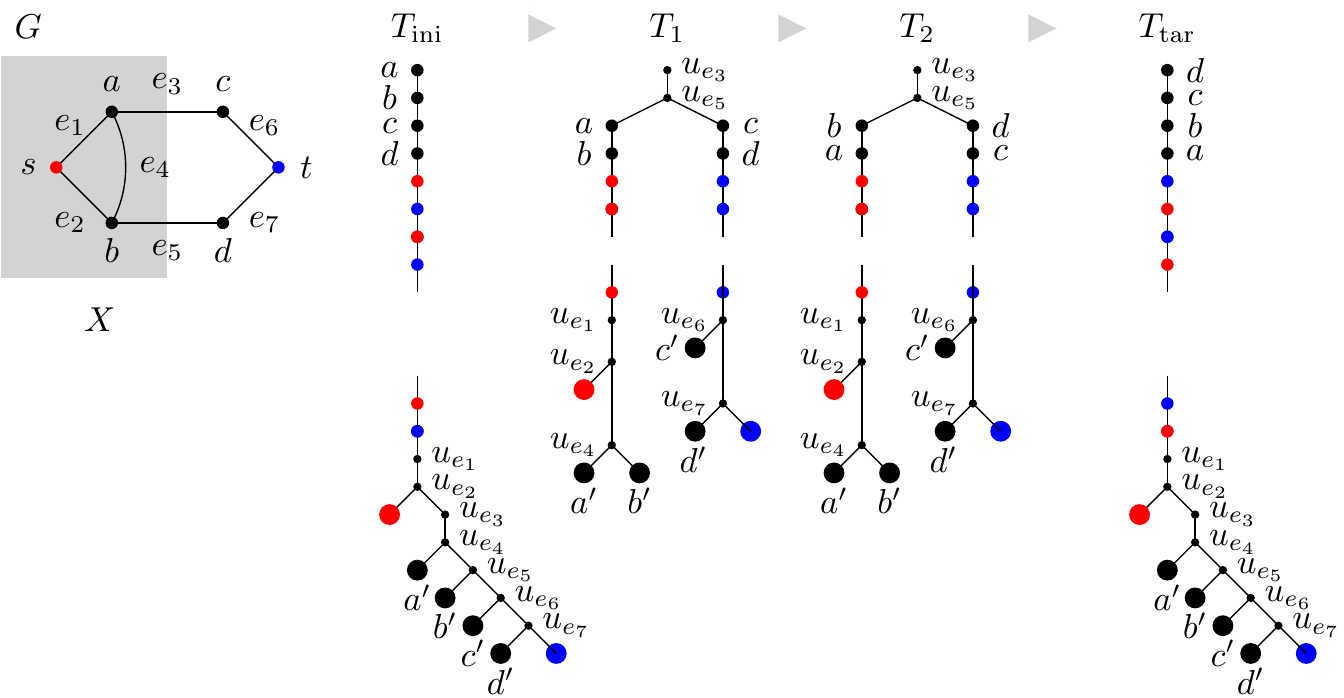}
\caption{A reconfiguration sequence from $T_\ini$ to $T_\tar$.}
\label{fig:reduction2}
\end{figure}

In $T_s$, by swapping $u$ and $v$ for every pair of $u, v \in X \setminus \{s\}$, 
we obtain an elimination tree in which $v_i$ is an ancestor of $v_j$ for $v_i, v_j \in X \setminus \{s\}$ with $i > j$. 
The total weight of these swaps is $\binom{ |X| -1}{2} \cdot N^2$. 
Similarly, by applying swaps with weight $\binom{ |V \setminus X| -1}{2} \cdot N^2$ to $T_t$, 
we obtain an elimination tree in which $v_i$ is an ancestor of $v_j$ for $v_i, v_j \in (V \setminus X) \setminus \{t\}$ with $i > j$.
Let $T_2$ be the elimination tree obtained from $T_1$ by applying these operations. 

Starting from $T_2$, we swap an element in $U$ and its child repeatedly so that 
we obtain an elimination tree $T_{\tar}$. 
This can be done by applying swaps whose total weight is at most $|U| ( 2n N + 2 N^7 + m)$. 

Therefore, the total weight of the above swaps from $T_{\ini}$ to $T_{\tar}$ is at most  
\begin{align*}
& 2 |U| ( 2n N + 2 N^7 + m) + \binom{|X|-1}{2} \cdot N^2 + \binom{|V \setminus X| -1}{2} \cdot N^2 \\ 
&\quad = 4 \lambda N^7 + n (n-1) N^2 + 4 \lambda n N + 2 \lambda m \\
&\quad < 4 \lambda N^7 + (n^2 - n +1) N^2, 
\end{align*}
where we note that $|U| = \lambda$ and $|X| = |V \setminus X| = n+1$. 
This shows the sufficiency.

\subsubsection*{Necessity (``only if'' part)}

Let $\mathbf{T}$ be a reconfiguration sequence from $T_{\ini}$ to $T_{\tar}$ whose weight is less than $4 \lambda N^7 + (n^2 - n +1) N^2$. 
Since this weight is less than $N^8$, we observe the following.
\begin{observation}
\label{obs:01}
For $v \in V$, vertex $v'$ is not swapped with other vertices in $\mathbf{T}$. 
For $i, j \in [N^3]$, none of $\mathbf{swap}(s_i, s_j)$, $\mathbf{swap}(t_i, t_j)$, $\mathbf{swap}(s_i, t_j)$, and $\mathbf{swap}(t_i, s_j)$ is applied in $\mathbf{T}$.  
\end{observation}
By Observation~\ref{obs:01}, we cannot swap $s_1$ and $t_1$ directly, and hence 
$\mathbf{T}$ contains an elimination tree $T^*$ 
in which $s_1$ and $t_1$ are incomparable. 
Then, there exists a vertex set $V^* \subseteq V(H)$ such that 
$s_1$ and $t_1$ are contained in different connected components of $H - V^*$, 
and each vertex in $V^*$ is an ancestor of $s_1$ and $t_1$ in $T^*$. 
By Observation~\ref{obs:01} again, $V^*$ does not contain $v'$ for $v \in V$, that is, $V^* \subseteq V \cup \{u_e \mid e \in E\}$. 
Note that removing $V^* \cap V$ does not affect the connectedness of $H$ since each vertex $v \in V^* \cap V$ has its copy $v'$ in $H$. 
Let
\[
  F := \{ e \in E \mid u_e \in V^* \}.
\]
Then, $s$ and $t$ are contained in different connected components of $G - F$, i.e., 
$F$ contains an $s$-$t$ cut set in $G$. 

Since removing $V$ does not affect the connectedness of $H$, we also observe the following.
\begin{observation}
\label{obs:03}
Let $T$ and $T'$ be elimination trees in $\mathbf{T}$ and let $e_1, e_2 \in E$ be distinct edges. 
If $u_{e_1} \in \mathbf{anc}_T(u_{e_2})$ and $u_{e_2} \in \mathbf{anc}_{T'}(u_{e_1})$, then 
$\mathbf{swap}(u_{e_3}, u_{e_4})$ is applied for some $e_3, e_4 \in E$ (possibly $\{e_1, e_2\} \cap \{e_3, e_4\} \neq \emptyset$) between $T$ and $T'$. 
\end{observation}

We divide $\mathbf{T}$ into two reconfiguration sequences $\mathbf{T}_1$ and $\mathbf{T}_2$, where 
$\mathbf{T}_1$ is from $T_{\ini}$ to $T^*$ and $\mathbf{T}_2$ is from $T^*$ to $T_{\tar}$. 
By symmetry, we may assume that 
\[
\mathrm{length}_w (\mathbf{T}_1) \le \frac{\mathrm{length}_w (\mathbf{T}) }{2} 
< 2 \lambda N^7 + N^3. 
\]
For $i \in [N^3]$, define 
\begin{align*}
L_i &= \{ e \in E \mid \mathbf{swap} (u_e, s_i) \text{ is applied in } \mathbf{T}_1 \}, \\
R_i &= \{ e \in E \mid \mathbf{swap} (u_e, t_i) \text{ is applied in } \mathbf{T}_1 \}. 
\end{align*}
For $i \in [N^3]$, let $\mathbf{swap}(L_i)$ denote the set of all swaps $\mathbf{swap}(u_e, s_i)$ in $\mathbf{T}_1$ with $e \in L_i$. 
Similarly, let $\mathbf{swap}(R_i)$ denote the set of all swaps $\mathbf{swap}(u_e, t_i)$ in $\mathbf{T}_1$ with $e \in R_i$.

\begin{claim}
\label{clm:04}
For $i \in [N^3]$, we have the following: 
\begin{itemize}
\item
if an edge $e \in E$ is contained in the connected component of $G-L_i$ containing $s$, then  
$u_e \in \mathbf{des}_T(s_i)$ for any elimination tree $T$ in $\mathbf{T}_1$, and
\item
if an edge $e \in E$ is contained in the connected component of $G-R_i$ containing $t$, then  
$u_e \in \mathbf{des}_T(t_i)$ for any elimination tree $T$ in $\mathbf{T}_1$. 
\end{itemize}
\end{claim}

\begin{claimproof}
For each edge $e \in E$ in the connected component of $G-L_i$ containing $s$, 
vertices $s_i$ and $u_e$ are contained in the same connected component in $H - \{u_f \mid f \in L_i\}$. 
Since 
$u_e \in \mathbf{des}_{T_{\ini}} (s_i)$ holds
 and $\mathbf{swap}(u_e, s_i)$ is not applied in $\mathbf{T}_1$ as $e \not\in L_i$, we have that $u_e \in \mathbf{des}_{T} (s_i)$ for any elimination tree $T$ in $\mathbf{T}_1$. 
The same argument works for the second statement. 
\end{claimproof}

To simplify the notation, let $L_0 = R_0 = F$. 
For $i \in [N^3] \cup \{0\}$, let $X_i \subseteq V$ be the vertex set of the connected component of $G-L_i$ containing $s$. 
Similarly, let $Y_i \subseteq  V$ be the vertex set of the connected component of $G-R_i$ containing $t$. 

\begin{claim}
\label{clm:05}
For $i, j \in [N^3] \cup \{0\}$ with $j > i$, we have the following: 
\begin{enumerate}[(i)]
\item
$(E[X_j] \setminus L_j) \cap L_i = \emptyset$, and 
\item
$(E[Y_j] \setminus R_j) \cap R_i = \emptyset$. 
\end{enumerate}
\end{claim}

\begin{claimproof}
To show (i), assume to the contrary that there exists $e \in (E[X_j] \setminus L_j) \cap L_i$ for some $j>i$. Note that $j \in [N^3]$. 
Since $e \in E[X_j] \setminus L_j$, Claim~\ref{clm:04} shows that 
$u_e \in \mathbf{des}_{T} (s_j)$ for any elimination tree $T$ in $\mathbf{T}_1$. 
If $i \ge 1$, then since $u_e \in \mathbf{des}_{T} (s_j)$ and $s_i \in \mathbf{anc}_{T} (s_j)$, we see that $u_e$ and $s_i$ are not adjacent in $T$. 
This implies that $\mathbf{swap}(u_e, s_i)$ is not applied in $\mathbf{T}_1$, which contradicts $e \in L_i$. 
If $i = 0$, then $e \in L_0 = F$ implies that $u_e \in \mathbf{anc}_{T^*} (s_1) \subseteq \mathbf{anc}_{T^*} (s_j)$, which contradicts $u_e \in \mathbf{des}_{T} (s_j)$ for any $T$. 
The same argument works for (ii). 
\end{claimproof}

\begin{claim}
\label{clm:09}
$X_0 \supseteq X_1 \supseteq X_2 \supseteq \dots \supseteq X_{N^3}$ and
$Y_0 \supseteq Y_1 \supseteq Y_2 \supseteq \dots \supseteq Y_{N^3}$. 
\end{claim}

\begin{claimproof}
Let $i, j \in [N^3] \cup \{0\}$ be indices with $j > i$. 
Since $(E[X_j] \setminus L_j) \cap L_i = \emptyset$ by Claim~\ref{clm:05} (i), all vertices in $X_j$ are contained in the same connected component of $G - L_i$. 
Since both $X_i$ and $X_j$ contain $s$, we obtain $X_j \subseteq X_i$. 
This shows that $X_0 \supseteq X_1 \supseteq X_2 \supseteq \dots \supseteq X_{N^3}$. 
Similarly, we obtain $Y_0 \supseteq Y_1 \supseteq Y_2 \supseteq \dots \supseteq Y_{N^3}$. 
\end{claimproof}

\begin{claim}
\label{clm:07}
For $i \in [N^3]$, we have $|L_i| = |R_i| = \lambda$, $L_i = \delta_G(X_i)$, and $R_i = \delta_G(Y_i)$. 
\end{claim}

\begin{claimproof}
Since $F$ contains an $s$-$t$ cut set, it holds that $X_0 \subseteq V \setminus \{t\}$. 
For $i \ge 1$, since $s \in X_i \subseteq X_0 \subseteq V \setminus \{t\}$ by Claim~\ref{clm:09}, we see that  $\delta_G(X_i)$ is an $s$-$t$ cut set contained in $L_i$. 
Similarly, $R_i$ contains an $s$-$t$ cut set in $G$. 
Therefore, we obtain $|L_i|, |R_i| \ge \lambda$ for any $i \in [N^3]$. 
By considering the weight of $\mathbf{T}_1$, we obtain 
\begin{align*}
2 \lambda  N^7 + N^3 &> \mathrm{length}_w (\mathbf{T}_1) \\ 
&\ge \sum_{i=1}^{N^3} (w(s_i) |L_i| + w(t_i) |R_i|) \\
&= 2 \lambda N^7 + N^4 \sum_{i=1}^{N^3} ( (|L_i| -\lambda) + (|R_i| -\lambda)), 
\end{align*}
which shows that $|L_i| = |R_i| = \lambda$ for any $i \in [N^3]$. 
Therefore, each of $L_i$ and $R_i$ is a minimum $s$-$t$ cut set in $G$, and hence $L_i = \delta_G(X_i)$ and $R_i = \delta_G(Y_i)$ hold. 
\end{claimproof}

Since the total weight of $\mathbf{swap} (L_i)$ and $\mathbf{swap}(R_i)$ is at least $2\lambda N^7$ by this claim, we see that
$u_e$ and $s_i$ (resp.~$t_i$) are swapped exactly once in $\mathbf{T}_1$ for $e \in L_i$ (resp.~$e \in R_i$) and for $i \in [N^3]$. 
For $i \in [N^3]$, let $T_i$ (resp.~$T'_i$) be the elimination tree that appears in $\mathbf{T}_1$ after all the swaps in $\mathbf{swap} (L_i)$ (resp.~$\mathbf{swap} (R_i)$) are just applied. 

\begin{claim}
\label{clm:11}
Elimination trees $T'_{N^3}, T_{N^3}, T'_{N^3-1}, T_{N^3-1}, \dots , T'_1, T_1$ appear in this order in $\mathbf{T}_1$. 
\end{claim}

\begin{claimproof}
We first show that $T_i$ appears after $T'_i$ for $i \in [N^3]$. 
Let $e \in R_i$ be the edge such that $\mathbf{swap}(u_e, t_i)$ is applied just before obtaining $T'_i$.
Ler $f \in L_i \setminus (R_i \setminus \{e\})$, where the existence of such $f$ is guaranteed by $|L_i|=|R_i|$. 
Since $R_i$ is a minimum $s$-$t$ cut set by Claim~\ref{clm:07}, we see that $G - (R_i \setminus \{e\})$ is connected. 
Then, for any elimination tree $T$ before $T'_i$, we have $u_{e'} \in \mathbf{des}_T(t_i)$ for any $e' \in E \setminus (R_i \setminus \{e\})$. 
In particular, $u_f \in \mathbf{des}_T(t_i)$. 
Since $s_i \in \mathbf{anc}_T(t_i)$, we see that $u_f$ and $s_i$ are not adjacent in $T$. 
This shows that  we cannot apply $\mathbf{swap} (u_{f}, s_i)$ before $T'_i$. 
Therefore, $T_i$ appears after $T'_i$ in $\mathbf{T}_1$.

By the same argument, we can show that $T'_i$ appears after $T_{i+1}$ for $i \in [N^3-1]$. 
Therefore, $T'_{N^3}$, $T_{N^3}$, $T'_{N^3-1}$, $T_{N^3-1}, \dots , T'_1$, $T_1$ appear in this order.  
\end{claimproof}

\begin{claim}
\label{clm:12}
For $i \in [N^3]$, we have the following: 
\begin{itemize}
\item
$u_e \in \mathbf{anc}_{T_i}(u_{e'})$ for any $e \in L_i$ and $e' \in E \setminus L_i$, and 
\item
$u_e \in \mathbf{anc}_{T'_i}(u_{e'})$ for any $e \in R_i$ and $e' \in E \setminus R_i$.  
\end{itemize}
\end{claim}

\begin{claimproof}
Let $T$ be an elimination tree in $\mathbf{T}_1$ just before $T_i$. Then, there exists an edge $f \in L_i$ such that 
$T_i$ is obtained from $T$ by applying $\mathbf{swap}(u_{f}, s_i)$. 
Since $G - (L_i \setminus \{f\})$ is connected by Claim~\ref{clm:07}, 
we have $u_e \in \mathbf{anc}_T(s_i)$ for $e \in L_i \setminus \{f\}$ and 
$u_{e'} \in \mathbf{des}_T(s_i)$ for $e' \in E \setminus (L_i \setminus \{f\})$. 
Therefore, after applying $\mathbf{swap}(u_{f}, s_i)$, we obtain
$u_e \in \mathbf{anc}_{T_i}(f)$ for $e \in L_i \setminus \{f\}$ and 
$u_{e'} \in \mathbf{des}_{T_i}(f)$ for $e' \in E \setminus L_i$. 
This shows that 
$u_e \in \mathbf{anc}_{T_i}(u_{e'})$ for any $e \in L_i$ and $e' \in E \setminus L_i$. 
By the same argument, we obtain  
$u_e \in \mathbf{anc}_{T'_i}(u_{e'})$ for any $e \in R_i$ and $e' \in E \setminus R_i$. 
\end{claimproof}

\begin{claim}
\label{clm:10}
$X_1 = V \setminus Y_1$. 
\end{claim}

\begin{claimproof}
Observe that $X_0$ and $Y_0$ are disjoint since $F = L_0 =R_0$ contains an $s$-$t$ cut set in $G$. 
Since $X_1 \subseteq X_0$ and $Y_1 \subseteq Y_0$ by Claim~\ref{clm:09}, we see that $X_1$ and $Y_1$ are disjoint. 
To derive a contradiction, 
assume that $X_1 \neq V \setminus Y_1$, that is, $X_1$ and $Y_1$ are disjoint sets with $X_1 \cup Y_1 \subsetneq V$. 
Then, by Claim~\ref{clm:09}, we obtain $X_i \neq V \setminus Y_i$ for any $i \in [N^3]$. 
This shows that $L_i \neq R_i$ for any $i \in [N^3]$. 
Since $|L_i| = |R_i| =\lambda$, there exist $f_i \in L_i \setminus R_i$ and $f'_i \in R_i \setminus L_i$. 
By Claim~\ref{clm:12}, 
we obtain $u_{f_i} \in \mathbf{anc}_{T_i}(u_{f'_i})$ and $u_{f'_i} \in \mathbf{anc}_{T'_i}(u_{f_i})$. 
By Observation~\ref{obs:03}, $\mathbf{swap}(u_e, u_{e'})$ is applied for some $e, e' \in E$ between $T'_i$ and $T_i$. 
Since such a swap is required for each $i \in [N^3]$, by Claim~\ref{clm:11}, 
we have to swap pairs in $\{u_e \mid e \in E\}$ at least $N^3$ times in $\mathbf{T}_1$. 
Therefore, we obtain 
\[
 \mathrm{length}_w (\mathbf{T}_1) \ge \sum_{i=1}^{N^3} (w(s_i) |L_i| + w(t_i) |R_i|)  + N^3 = 2 \lambda N^7 + N^3, 
\]
which contradicts $\mathrm{length}_w (\mathbf{T}_1) <  2 \lambda N^7 + N^3$. 
\end{claimproof}

\begin{claim}
\label{clm:13}
$F = \delta_G(X_1) = \delta_G(Y_1)$. 
\end{claim}

\begin{claimproof}
Claims~\ref{clm:07} and~\ref{clm:10} show that $L_1= R_1 = \delta_G(X_1) = \delta_G(Y_1)$. 
This together with Claim~\ref{clm:05} shows that 
$F \cap E[X_1] = F \cap E[Y_1] = \emptyset$. 
Since $F$ contains an $s$-$t$ cut set in $G$, we obtain $F = \delta_G(X_1) = \delta_G(Y_1)$. 
\end{claimproof}

\begin{claim}
\label{clm:14}
Let $T$ be an elimination tree in $\mathbf{T}_1$. 
If two vertices $u, v \in V \setminus \{s, t\}$ are contained in the same connected component in $G - F$, 
then $u$ and $v$ are comparable in $T$. 
\end{claim}

\begin{claimproof}
By Claim~\ref{clm:13}, $G - F$ consists of two connected components $G[X_1]$ and $G[Y_1]$. 
We first consider the case when $u, v \in X_1 \setminus \{s\}$. 
By Claims~\ref{clm:04} and~\ref{clm:13},  
we obtain $u_e \in \mathbf{des}_T(s_1)$ for any $e \in E[X_1]$. 
Furthermore, 
since $\mathrm{length}_w(\mathbf{T}_1) < 2 \lambda N^7 + N^3$ holds and the total weight of $\mathbf{swap} (L_i)$ and $\mathbf{swap}(R_i)$ is $2\lambda N^7$, 
neither $\mathbf{swap}(s_1, u)$ nor $\mathbf{swap}(s_1, v)$ is applied in $\mathbf{T}_1$, because  $w(s_1) w(u) = w(s_1) w(v) = N^5$. 
Therefore, we obtain $u \in \mathbf{anc}_T(s_1)$ and $v \in \mathbf{anc}_T(s_1)$, which shows that 
 $u$ and $v$ are comparable in $T$. 
The same argument works when $u, v\in Y_1 \setminus \{t\}$. 
\end{claimproof}

Since the weight of $\mathbf{T}_1$ is at least $\sum_{i=1}^{N^3} (w(s_i) |L_i| + w(t_i) |R_i|) = 2 \lambda N^7$, we obtain 
\[
\mathrm{length}_w (\mathbf{T}_2) =  \mathrm{length}_w (\mathbf{T}) - \mathrm{length}_w (\mathbf{T}_1) < 2 \lambda N^7 + N^3. 
\]
Hence, the above argument (Claims~\ref{clm:04}--\ref{clm:14}) can be applied also to the reverse sequence of $\mathbf{T}_2$. 
In particular, Claim~\ref{clm:14} holds even if $\mathbf{T}_1$ is replaced with $\mathbf{T}_2$. 
Therefore, 
if two vertices $u, v \in V\setminus \{s, t\}$ are contained in the same connected component in $G - F$, 
then $u$ and $v$ are comparable in any elimination tree in $\mathbf{T}$. 
For such a pair of vertices $u$ and $v$, the only way to reverse the ordering of $u$ and $v$ is to apply $\mathbf{swap}(u, v)$ or $\mathbf{swap}(v, u)$. 

Recall that $G-F$ consists of two connected components $G[X_1]$ and $G[Y_1]$ by Claim~\ref{clm:13}. 
Since the ordering of $v_1, \dots , v_{2n}$ are reversed from $T_{\ini}$ to $T_{\tar}$, 
we see that $\mathbf{swap}(u, v)$ or $\mathbf{swap}(v, u)$ has to be applied in $\mathbf{T}$ if $u, v \in X_1 \setminus \{s\}$ or $u, v \in Y_1 \setminus \{t\} = (V \setminus X_1) \setminus \{t\}$. 
Furthermore, we have to swap some elements in $\{u_e \mid e \in E\}$ and $\{s_1, t_1, \dots , s_{N^3}, t_{N^3}\}$ in $\mathbf{T}_2$, 
whose total weight is at least $2\lambda N^7$ in the same way as $\mathbf{T}_1$.
With these observations, 
we evaluate the weight of $\mathbf{T}$ as follows, where we denote $k = |X_1|$ to simplify the notation:  
\begin{align*}
\mathrm{length}_w (\mathbf{T}) 
&\ge 2\lambda N^7 + 2\lambda N^7 + \binom{|X_1| -1}{2} \cdot N^2 + \binom{|V \setminus X_1| - 1}{2} \cdot N^2 \\
&= 4 \lambda N^7 + \frac{(k-1)(k-2)}{2} N^2 + \frac{(2n-k+1)(2n-k)}{2} N^2 \\
&= 4 \lambda N^7 + (k^2 - 2(n+1) k + 2n^2 + n+1) N^2 \\
&= 4 \lambda N^7 +  (n^2 - n) N^2 + (k - n - 1)^2 N^2.  
\end{align*}
This together with $\mathrm{length}_w (\mathbf{T})  < 4\lambda N^7 + (n^2 - n +1) N^2$ shows that   
$(k - n - 1)^2 < 1$, and hence $k=n+1$ by the integrality of $k$ and $n$. 

Therefore, we obtain $|X_1| = k = n+1$ and $|Y_1| = |V \setminus X_1| = n+1$. 
Since $|\delta_G(X_1)| = |L_1| = \lambda$, this shows that $X_1$ is a desired $s$-$t$ cut in $G$.

\section{Hardness of the Unweighted Problem (Proof of Theorem~\ref{thm:hardness})}
\label{sec:hardnessunweighted}

To show Theorem~\ref{thm:hardness}, we reduce \textsc{Weighted Combinatorial Shortest Path on Graph Associahedra} to \textsc{Combinatorial Shortest Path on Graph Associahedra}. 
An operation called \emph{projection} (e.g.\ \cite{defga-CardinalPVP}) plays an important role in our validity proof. 

\subsection{Useful Operation: Projection}

Let $G=(V, E)$ be a graph and let $T$ be an elimination tree associated with $G$. 
For $U \subseteq V$ such that $G[U]$ is connected, let  $T|_U$ be the elimination tree associated with $G[U]$ that preserves the ordering in $T$. 
That is, $u \in \mathbf{anc}_{T|_U} (v)$ if and only if $u \in \mathbf{anc}_{T} (v)$ and 
$u$ are $v$ are connected in $G[U] - \mathbf{anc}_{T} (u)$ for $u, v \in U$. 
Note that such $T|_U$ is uniquely determined. 
We call $T|_U$ the \emph{projection} of $T$ to $U$. 
See \figurename~\ref{fig:restriction1} for illustration.

\begin{figure}[t]
    \centering
    \includegraphics{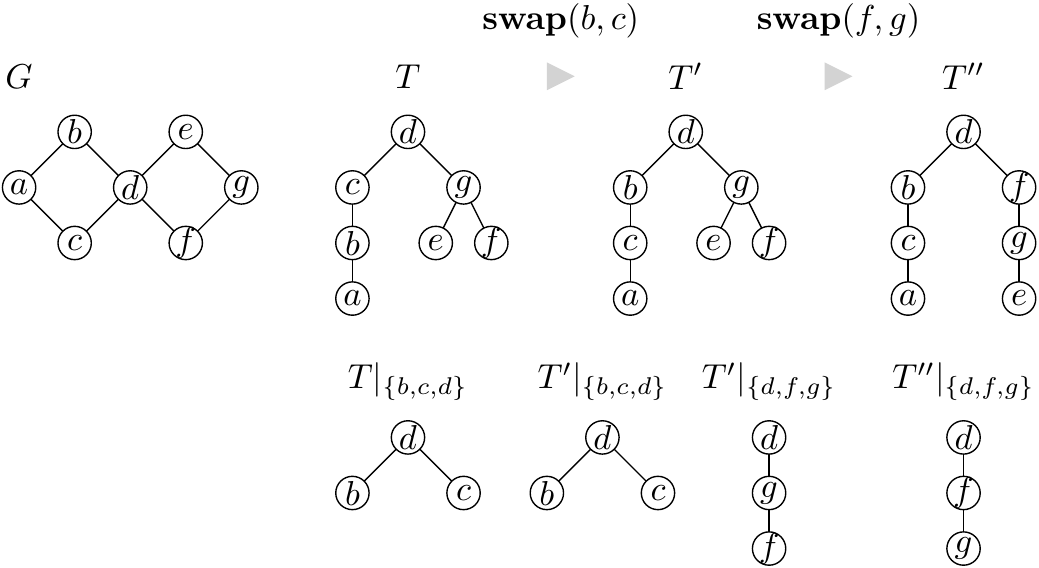}
    \caption{An example of projections. Note that $T|_{\{b,c,d\}} = T'|_{\{b,c,d\}}$ since $b$ and $c$ are incomparable in $T|_{\{b,c,d\}}$, and $T''|_{\{d,f,g\}}$ is obtained from $T'|_{\{d,f,g\}}$ by swapping $f$ and $g$ since $f$ and $g$ are adjacent in $T'|_{\{d,f,g\}}$.}
    \label{fig:restriction1}
\end{figure}

\begin{lemma}
\label{lem:restriction}
Let $U \subseteq V$ be a vertex set such that $G[U]$ is connected. 
Let $T$ and $T'$ be elimination trees associated with $G$ such that 
$T'$ is obtained from $T$ by applying $\mathbf{swap}(u, v)$, where $u, v \in V$. 
\begin{enumerate}
\item If $\{u, v\} \subseteq U$, then either $T'|_U = T|_U$ or $T'|_U$ is obtained from $T|_U$ by applying $\mathbf{swap}(u, v)$. 
\item Otherwise, $T'|_U = T|_U$. 
\end{enumerate}
\end{lemma}

\begin{proof}
Since all the vertices in $V \setminus U$ are removed when we construct $T|_U$, 
$\mathbf{swap}(u, v)$ affects $T|_U$ only if $\{u, v\} \subseteq U$, which proves the second item. 
For the first item, suppose that $\{u, v\} \subseteq U$. 
Then, $u$ and $v$ are adjacent or incomparable in $T|_U$. 
If they are adjacent, then $T'|_U$ is obtained from $T|_U$ by applying $\mathbf{swap}(u, v)$. 
If they are incomparable, then  $T'|_U = T|_U$. 
\end{proof}

\subsection{Reduction}

Suppose we are given 
a graph $G=(V, E)$, two elimination trees $T_{\ini}$ and $T_{\tar}$, 
and a weight function $w \colon V \to \mathbb{Z}_{>0}$, which form an instance of  \textsc{Weighted combinatorial Shortest Path on Graph Associahedra}. 
Then, we replace each vertex $v \in V$ with a clique of size $w(v)$. 
Formally, consider a graph $G'=(V', E')$ such that
$V' = \{ v_i \mid v \in V,\ i \in \{1, \dots , w(v)\} \}$, and 
$\{u_i, v_j\} \in E'$ if $\{u, v\} \in E$ or $u=v$. 
Let $T'_{\ini}$ (resp.~$T'_{\tar}$) be the elimination tree obtained from $T_{\ini}$ (resp.~$T_{\tar}$) by replacing a vertex $v \in V$ with a path $v_1, v_2, \dots , v_{w(v)}$. 
That is, for distinct $u, v \in V$, there is an arc $(u, v)$ in $T_{\ini}$ (resp.~$T_{\tar}$) if and only if $(u_{w(u)}, v_1)$ is an arc of $T'_{\ini}$ (resp.~$T'_{\tar}$). 
Note that the obtained elimination tree is associated with $G'$. 
This defines an instance of \textsc{Combinatorial Shortest on Graph Associahedra}.

\subsection{Validity}
In what follows, we show that the 
obtained instance of  \textsc{Combinatorial Shortest Path on Graph Associahedra} has a reconfiguration sequence of length at most $\ell$ if and only if 
the original instance of \textsc{Weighted Combinatorial Shortest Path on Graph Associahedra} has a reconfiguration sequence $\mathbf{T}$ with $\mathrm{length}_w (\mathbf{T}) \le \ell$.

\subsubsection*{Sufficiency (``if'' part)}

Suppose that the original instance of \textsc{Weighted Combinatorial Shortest Path on Graph Associahedra} has a reconfiguration sequence $\mathbf{T}$ from $T_\ini$ to $T_\tar$. 
Then, we construct a reconfiguration sequence $\mathbf{T}'$ from $T'_\ini$ to $T'_\tar$
by replacing each swap $\mathbf{swap}(u, v)$ in $\mathbf{T}$ with $w(u) \cdot w(v)$ swaps $\{\mathbf{swap}(u_i, v_j) \mid i \in [w(u)],\ j \in [w(v)]\}$.  
This gives a reconfiguration sequence from $T'_\ini$ to $T'_\tar$ whose length is $\mathrm{length}_w (\mathbf{T})$, which shows the sufficiency.

\subsubsection*{Necessity (``only if'' part)}

Suppose that the obtained instance of  \textsc{Combinatorial Shortest Path on Graph Associahedra} has a reconfiguration sequence $\mathbf{T}'$ from $T'_{\ini}$ to $T'_{\tar}$ of length at most $\ell$. 
For any $v \in V$, since $v_1, \dots , v_{w(v)}$ form a clique, they are comparable in any elimination tree in $\mathbf{T}'$. Furthermore, since $v_1, \dots , v_{w(v)}$ are aligned in this order in both $T'_{\ini}$ and $T'_{\tar}$, 
we may assume that $\mathbf{swap}(v_i, v_j)$ is not applied in $\mathbf{T}'$ for any $i, j \in [w(v)]$. 

Let $\Phi$ be the set of all maps $\phi \colon V \to \mathbb{Z}$ such that $\phi(v) \in \{1, \dots , w(v)\}$ for any $v \in V$. Note that $|\Phi| = \prod_{v \in V} w(v)$. 
For $\phi \in \Phi$, define $U_\phi = \{v_{\phi(v)} \mid v \in V\}$. 
Note that $G'[U_\phi]$ is isomorphic to $G$, and hence it is connected. 
By projecting each elimination tree in $\mathbf{T}'$ to $U_\phi$, we obtain a sequence of elimination trees. 
Lemma~\ref{lem:restriction} shows that this forms a reconfiguration sequence, say $\mathbf{T}_\phi$, if we remove duplications when the same elimination tree appears consecutively. 
Since $G'[U_\phi]$ is isormorphic to $G$, by idenfitying $v_{\phi(v)}$ with $v$ for each $v\in V$, we can regard $\mathbf{T}_\phi$ as a recofiguration sequence from $T_\ini$ to $T_\tar$.  
That is, $\mathbf{T}_\phi$ is regarded as a feasible solution of the original instance of \textsc{Weighted Combinatorial Shortest Path on Graph Associahedra}. 

In what follows, we consider reconfiguration sequences $\{\mathbf{T}_\phi \mid \phi \in \Phi\}$ and show that 
a desired sequence exists among them. 
Suppose that $\mathbf{swap}(u_i, v_j)$ is applied in $\mathbf{T}'$, where $u, v \in V$, $i \in [w(u)]$, and $j \in [w(v)]$. 
Then, Lemma~\ref{lem:restriction} shows that the corresponding swap operation $\mathbf{swap}(u_i, v_j)$, which is identified with $\mathbf{swap}(u, v)$, is applied in $\mathbf{T}_\phi$ only if  $\phi(u) =i$ and $\phi(v)=j$. 
Thus, such a swap is applied in at most $| \{\phi \in \Phi \mid \phi(u)=i,\ \phi(v)=j\} | = |\Phi | / (w(u) \cdot w(v))$ sequences in $\{\mathbf{T}_\phi \mid \phi \in \Phi\}$. 
Therefore, we obtain 
\begin{align*}
\sum_{\phi \in \Phi} \mathrm{length}_w (\mathbf{T}_\phi) 
&= \sum_{\phi \in \Phi} \sum_{\mathbf{swap}(u, v) \in \mathbf{T}_\phi} w(\mathbf{swap}(u, v)) \\ 
&\le \sum_{\mathbf{swap}(u_i, v_j) \in \mathbf{T}'} w(\mathbf{swap}(u, v)) \cdot  \frac{|\Phi |}{w(u) \cdot w(v)}\\
&= \mathrm{length} (\mathbf{T}') \cdot |\Phi | \le \ell \cdot |\Phi |,  
\end{align*}
where each reconfiguration sequence is regarded as a multiset of swaps.  
Therefore, 
\[
\min_{ \phi \in \Phi} (\mathrm{length}_w (\mathbf{T}_\phi)) \le \frac{1}{|\Phi|} \sum_{\phi \in \Phi} \mathrm{length}_w (\mathbf{T}_\phi) \le \ell. 
\]
Hence, there exists $\phi \in \Phi$ such that $\mathbf{T}_\phi$ is a desired sequence. 
This shows the necessity.

Therefore, the weighted problem can be reduced to the unweighted problem, and hence Theorem~\ref{thm:hardnessweighted} implies Theorem~\ref{thm:hardness}.

\section{Hardness for Polymatroids
(Proof of Theorem~\ref{thm:hardness-polymatroid})}
\label{sec:polymatroid}

In this section, we prove Theorem~\ref{thm:hardness-polymatroid}. 
We reduce 
\textsc{Combinatorial Shortest Path on Graph Associahedra}
to 
\textsc{Combinatorial Shortest Path on Polymatroids}. 
Assume that we are given an instance 
$G=(V,E)$, $T_{\ini}$, and 
$T_{\tar}$ of 
\textsc{Combinatorial Shortest Path on Graph Associahedra}. 
To this end, we 
construct a polymatroid $(V,f)$ 
satisfying the following conditions. 
\begin{enumerate}
\item
$\mathbf{B}(f)$ is a realization of 
the $G$-associahedron.
\item
For each subset $X \subseteq V$, 
we can evaluate the value $f(X)$ in time 
bounded by a polynomial in the size of $G$.
\item 
We can find the extreme points $x_\ini, x_\tar$
of $\mathbf{B}(f)$
corresponding to $T_{\ini}, T_{\tar}$, respectively, 
in time 
bounded by a polynomial in the size of $G$. 
\end{enumerate}

We first argue that the conditions above suffice for our proof.
Suppose the existence of a polymatroid $(V, f)$ with the properties above.
Then, we may construct a polynomial-time algorithm for \textsc{Combinatorial Shortest Path on Graph Associahedra} with a fictitious polynomial-time algorithm for \textsc{Combinatorial Shortest Path on Polymatroids} as follows.
Let $(G, T_\ini, T_\tar)$ be an instance of \textsc{Combinatorial Shortest Path on Graph Associahedra}.
From Properties 1 and 3, we can construct an instance $((V, f), x_\ini, x_\tar)$ of \textsc{Combinatorial Shortest Path on Polymatroids} in polynomial time.
By the fictitious polynomial-time algorithm, we can solve the instance in time bounded by a polynomial in $|V|$ and the number of oracle calls to $f$.
By Property 2, this running time is bounded by a polynomial in $|V|$.
Thus, we find a solution to $(G, T_\ini, T_\tar)$ in polynomial time, and the proof is completed.

In our construction of such a polymatroid $(V,f)$, 
we use the realization of the $G$-associahedron by Devadoss~\cite{DBLP:journals/dm/Devadoss09}, which can be described as follows. 

Let $G = (V,E)$ be a connected undirected graph
with at least two vertices. We 
define $n := |V|$. 
For each subset $X \subseteq V$,
we call  
a connected component\footnote{%
More precisely speaking, $C$ is the vertex set of a connected component.}
$C$ of $G - X$ 
a  \emph{non-trivial} 
(resp.\ \emph{trivial}) connected component 
if $|C| \ge 2$ (resp.\ $|C|=1$). 
For each subset $X \subseteq V$, 
we define $\mathcal{C}^{\ast}(X)$ as the set of non-trivial 
connected components 
of $G - X$.

Let $T$ be an elimination tree of $G$.
For each vertex $v \in V$, we define 
$T(v)$ as the vertex set of the subtree of $T$ rooted at $v$. 
Then, we define the vector $x^T \in \mathbb{R}^V$ by choosing the coordinate $x^T(v)$ at every vertex of $v$
from the leaves to the root according to the following rule. 
\begin{itemize}
\item
If $v$ is a leaf of $T$, 
then we define $x^T(v) := 0$. 
\item
If $v$ is not a leaf of $T$, 
then we define $x^T(v)$ so that 
\[
\sum_{u \in T(v)}x^T(u) = 3^{|T(v)| - 2}. 
\] 
\end{itemize}
Define 
$\mathcal{E} := \{x^T \in \mathbb{R}^V \mid \text{$T$ is an elimination tree of $G$}\}$. 

\begin{lemma}[Devadoss~\cite{DBLP:journals/dm/Devadoss09}] \label{lemma:D09}
The convex hull of 
$\mathcal{E}$ is a realization of 
the $G$-associahedron.
For each elimination tree $T$ of $G$, 
the point $x^T$ is an extreme point of
the $G$-associahedron.
\end{lemma}

Define the function $f \colon 2^V \to \mathbb{R}$ by
\begin{equation} \label{eq:definition}
f(X) := 
\displaystyle{3^{n-2} - \sum_{C \in \mathcal{C}^{\ast}(X)}3^{|C|-2}}
\end{equation}
for each subset $X \subseteq V$. 
Notice that since $n \ge 2$ and 
$G$ is connected, 
$f(\emptyset)=0$. 
It is not difficult to see that 
we can evaluate the values of the function $f$ in time bounded by 
a polynomial in the size of $G$. 

\begin{lemma} \label{lemma:monotone}
Let $a_1,a_2,\dots,a_k$ be positive integers. Then
\begin{equation*}
3^{a_1+a_2+\dots+a_k} \ge 3^{a_1} + 3^{a_2} + \dots + 3^{a_k}. 
\end{equation*}
\end{lemma}
\begin{proof}
When $k = 1$, this statement holds by definition.

In what follows, we prove that for every integer $k \ge 2$,
\begin{equation} \label{eq:integer}
3^{a_1+a_2+\dots+a_k} - (3^{a_1} + 3^{a_2} + \dots + 3^{a_k}) \ge 3. 
\end{equation}
First, we consider the case where $k = 2$. In this case, 
since $a_1,a_2 \ge 1$, we have 
\[
3^{a_1 + a_2} - (3^{a_1} + 3^{a_2})
= 
(3^{a_1} - 1)(3^{a_2} - 1) - 1
\ge 3. 
\]
Let $\ell \geq 2$ be a positive integer, and assume that \eqref{eq:integer} holds when $k = \ell$. 
Then, we consider the case where $k = \ell + 1$. 
Since $a_1,a_2,\dots,a_{\ell+1}$ are positive integers, we have 
\begin{equation*}
\begin{split}
& 3^{a_1+a_2+\dots+a_{\ell+1}} - (3^{a_1} + 3^{a_2} + \dots + 3^{a_{\ell+1}})\\
& = 3^{a_{\ell+1}} \cdot 3^{a_1+a_2+\dots+a_{\ell}} 
- (3^{a_1} + 3^{a_2} + \dots + 3^{a_{\ell+1}})\\
& \ge 3^{a_{\ell+1}} (3^{a_1} + 3^{a_2} + \dots + 3^{a_{\ell}} + 3) 
- (3^{a_1} + 3^{a_2} + \dots + 3^{a_{\ell+1}})\\
& = (3^{a_1 + a_{\ell+1}} + 3^{a_2 + a_{\ell+1}} + \dots + 3^{a_{\ell} + a_{\ell+1}} 
+ 3^{a_{\ell+1}+1}) 
- (3^{a_1} + 3^{a_2} + \dots + 3^{a_{\ell+1}})\\
& \ge 3^{a_{\ell+1}}(3 - 1) \ge 3, 
\end{split}
\end{equation*}
where the first inequality follows from the induction hypothesis. 
This completes the proof. 
\end{proof} 

\begin{lemma} \label{lemma1:polymatroid}
The function $f$ defined in \eqref{eq:definition} 
satisfies \textrm{(P1)}, \textrm{(P2)}, and \textrm{(P3)}.
\end{lemma} 
Here, we recall the three properties:
(P1) $f(\emptyset) = 0$;
(P2) If $X \subseteq Y \subseteq V$, then $f(X) \le f(Y)$;
(P3) If $X, Y \subseteq V$, then $f(X \cup Y) + f(X \cap Y) \le f(X) + f(Y)$.

\begin{proof}
The property (P1) readily follows from the definition of $f$.
Thus, we consider (P2) and (P3). 

We first prove (P2). 
Let $X$ be a subset $V$, and let 
$v$ be a vertex in $V \setminus X$. 
It is sufficient to prove that 
$f(X) \le f(X \cup \{v\})$. 
Let $C_1,C_2,\dots,C_k$ be the non-trivial 
connected components 
of $G - X$. 
Then we have 
\begin{equation*}
f(X) = 3^{n-2} - \sum_{i = 1}^k3^{|C_i|-2}.
\end{equation*}
If $v \notin \bigcup_{i=1}^kC_i$, then $v$ is an isolated vertex of 
$G - X$. This implies that ${\cal C}^{\ast}(X) = {\cal C}^{\ast}(X \cup \{v\})$. 
Thus, $f(X) = f(X \cup \{v\})$. 
From here, we assume that $v \in \bigcup_{i=1}^kC_i$. 
Without loss of generality, we assume that 
$v \in C_k$. 
Then, $C_k \setminus \{v\}$ is partitioned into connected components
$C^{\prime}_1,C^{\prime}_2,\dots,C^{\prime}_{\ell}$ 
of $G - (X \cup \{v\})$. 
Assume that  
$C^{\prime}_1,C^{\prime}_2,\dots,C^{\prime}_t$ are non-trivial, and 
$C^{\prime}_{t+1},C^{\prime}_{t+2},\dots,C^{\prime}_{\ell}$ are trivial. 
Then, we have 
\[
f(X \cup \{v\}) = 3^{n-2} - \sum_{i = 1}^{k-1}3^{|C_i|-2} - 
\sum_{i = 1}^t3^{|C^{\prime}_i|-2}.
\]
Since $|C_k| > \sum_{i = 1}^t|C^{\prime}_i|$, 
we have 
\begin{equation*}
\begin{split}
f(X\cup \{v\}) - f(X) = 
3^{|C_k|-2} - \sum_{i = 1}^t3^{|C_i|-2}
\ge 
3^{\sum_{i = 1}^t|C^{\prime}_i|-2} - \sum_{i = 1}^t3^{|C^{\prime}_i|-2}
\ge 0,
\end{split}
\end{equation*}
where the first inequality follows from the monotonicity of the function 
$g(x) = 3^{x-2}$, and 
the second inequality follows from Lemma~\ref{lemma:monotone}. 

Next, we prove (P3). 
Let $X$ be a subset of $V$.
Let 
$u,v$ be distinct vertices in $V \setminus X$. 
Then it is sufficient to prove that 
\begin{equation*}
f(X\cup \{u\}) + f(X \cup \{v\}) \ge 
f(X \cup \{u,v\}) + f(X).
\end{equation*}
(See e.g., \cite[p.58]{F05}.)
Let $C_1,C_2,\dots,C_k$ be the non-trivial 
connected components 
of $G - X$. 
Then, we have 
\[
f(X) = 3^{n-2} - \sum_{i = 1}^{k} 3^{|C_i|-2}.
\]

First, we consider the case where $u \notin \bigcup_{i=1}^kC_i$. 
In this case, since $u$ is an isolated vertex in $G - X$, 
we have $f(X) = f(X\cup \{u\})$ 
and $f(X \cup \{v\}) = f(X \cup \{u,v\})$.
Similarly, we can treat the case where $v \notin \bigcup_{i=1}^kC_i$. 

Next, we consider the case where $u,v \in \bigcup_{i=1}^{k} C_i$,
and $u,v$ belong to distinct connected components 
of $G - X$. 
Without loss of generality, we assume that 
$u \in C_{k-1}$ and 
$v \in C_k$. 
\begin{itemize}
\item
Let $C^{\circ}_1,C^{\circ}_2,\dots,C^{\circ}_{\ell}$ be 
the non-trivial connected components of $G[C_{k-1}] - \{u\}$, where $G[C_{k-1}]$ denotes the subgraph of $G$ induced by $C_{k-1}$.
\item
Let $C^{\bullet}_1,C^{\bullet}_2,\dots,C^{\bullet}_{t}$ be 
the non-trivial connected components of $G[C_k] - \{v\}$.
\end{itemize}
Then, the non-trivial connected components of
$G[C_{k-1}\cup C_k] - \{u,v\}$ are 
$C^{\circ}_1, \dots, C^{\circ}_{\ell}, C^{\bullet}_1, \dots, C^{\bullet}_{t}$.
Thus, we have 
\[
f(X\cup \{u\}) + f(X \cup \{v\}) = 
f(X \cup \{u,v\}) + f(X).
\]

Finally, we consider the case where 
$u,v \in \bigcup_{i=1}^{k} C_i$,
and $u,v$ belong to the same connected component 
of $G - X$. 
Without loss of generality, we assume that 
$u,v \in C_k$. 
\begin{itemize}
\item
Let $C^{\circ}_1,C^{\circ}_2,\dots,C^{\circ}_{\ell}$ be 
the non-trivial connected components of $G[C_k] - \{u\}$. 
\item
Let $C^{\bullet}_1,C^{\bullet}_2,\dots,C^{\bullet}_{t}$ be 
the non-trivial connected components of $G[C_k] - \{v\}$.
\item
Let $C^{\prime}_1,C^{\prime}_2,\dots,C^{\prime}_{m}$ be 
the non-trivial connected components of
$G[C_k] - \{u,v\}$. 
\end{itemize}
Then,
since 
$\sum_{i=1}^{\ell}|C_i^{\circ}| \le |C_k|-1$ and
$\sum_{i=1}^t|C_i^{\bullet}| \le |C_k|-1$,
we have 
\begin{equation*}
\begin{split}
& f(X\cup \{u\}) + f(X \cup \{v\}) - 
f(X \cup \{u,v\}) - f(X)\\
& = 
{} - \sum_{i = 1}^{\ell}3^{|C^{\circ}_i|-2}
- \sum_{i = 1}^t3^{|C^{\bullet}_i|-2}
+ 3^{|C_k|-2}
+ \sum_{i = 1}^m3^{|C^{\prime}_i|-2}\\
& \ge 
{} - \sum_{i = 1}^{\ell}3^{|C^{\circ}_i|-2}
- \sum_{i = 1}^t3^{|C^{\bullet}_i|-2}
+ 3^{|C_k|-2}\\
& \ge 
{} - 3^{\sum_{i = 1}^{\ell}|C^{\circ}_i|-2}
- 3^{\sum_{i = 1}^t|C^{\bullet}_i|-2}
+ 3^{|C_k|-2}\\
& \ge 
{} - 3^{|C_k| - 3}
- 3^{|C_k| - 3}
+ 3^{|C_k|-2}
= 3^{|C_k|-3} \ge 0, 
\end{split}
\end{equation*}
where the second inequality follows from Lemma~\ref{lemma:monotone}.
This completes the proof. 
\end{proof} 

Let $(U, \rho)$ be a polymatroid.
It is well known that, for every vector $x \in \mathbb{R}^U$, 
$x$ is an extreme point of $\mathbf{B}(\rho)$ if and only if there exists an 
ordering $\sigma = (u_1,u_2,\dots,u_{|U|})$ of the elements in $U$ such that 
\[
x(u_i) = \rho(\{u_1,u_2,\dots,u_i\}) - \rho(\{u_1,u_2,\dots,u_{i-1}\})
\]
for every integer $i \in [|U|]$ (see \cite{E70,L83,S71}). 
In this case, we say that 
\emph{$x$ is generated by $\sigma$}. 

To prove that $x^T$ is an extreme point of $\mathbf{B}(f)$ for each elimination tree of $G$, we construct an ordering $\sigma$ on $V$ such that $x^T$ is generated by $\sigma$.
To this end, we say 
the elimination tree $T$ is \emph{compatible with}
an ordering $\sigma = (v_1,v_2,\dots,v_n)$ of the vertices 
in $V$ if
for every pair of integers $i,j \in [n]$ such that 
$i \neq j$, 
if $v_i$ is an ancestor of $v_j$ in $T$, then 
$i < j$.  
In other words, $T$ is compatible with $\sigma$ if $\sigma$ is a linear extension of $T$ when $T$ is seen as a partially ordered set.

\begin{lemma} \label{lemma2:polymatroid}
Let $\sigma = (v_1,v_2,\dots,v_n)$ be an ordering of the vertices in 
$V$.
Let $z$ be an extreme point of $\mathbf{B}(f)$ generated 
by $\sigma$.
In addition, let $T$ be an elimination tree of $G$ compatible 
with $\sigma$.
Then, $z = x^T$. 
\end{lemma}
\begin{proof}
Let $k$ be an integer in $[n]$.
We prove 
that $z(v_k) = x^T(v_k)$.
Let $C$ be the connected component of $G - \{v_1,v_2,\dots,v_{k-1}\}$
containing $v_k$.  
If $v_k$ is a leaf of $T$, then $|C|=1$. 
Thus, we have $z(v_k) = x^T(v_k) = 0$. 
Assume that $v_k$ is not a leaf of $T$.
Notice that, in this case, $|C| \ge 2$. 
Let $C_1,C_2,\dots,C_{\ell}$ be 
the connected components 
of $G[C] - \{v_{k}\}$.
Assume that 
$C_1, C_2, \dots, C_t$ are non-trivial, 
and 
$C_{t+1}, C_{t+2}, \dots, C_{\ell}$ are trivial. 
For each integer $i \in [\ell]\setminus [t]$, 
we assume that $C_i = \{u_i\}$. 
Let $c_1,c_2,\dots,c_{\ell}$ be the children of $v_k$ in $T$.
We assume that 
$c_i \in C_i$ for every integer $i \in [\ell]$. 
Then, for every integer $i \in [t]$, 
$T(c_i) = C_i$, which implies that 
\[
\sum_{u \in T(c_i)}x^T(u) = 3^{|C_i|-2}.
\]
We have 
\[
z(v_k) = f(\{v_1,v_2,\dots,v_k\}) - f(\{v_1,v_2,\dots,v_{k-1}\})
= 3^{|C|-2} - \sum_{i = 1}^t3^{|C_i|-2}.
\]
Furthermore, since 
$u_i$ is a leaf of $T$ for every integer $i \in [\ell]\setminus [t]$, 
\begin{equation*}
x^T(v_k) 
= 3^{|C|-2} - \sum_{i = 1}^{\ell}\sum_{u \in T(c_i)}x^T(u)
= 3^{|C|-2} - \sum_{i = 1}^t\sum_{u \in T(c_i)}x^T(u)
= 3^{|C|-2} - \sum_{i = 1}^t3^{|C_i|-2}. 
\end{equation*}
Thus, $z(v_k) = x^T(v_k)$. 
This completes the proof. 
\end{proof}

Recall that $\mathcal{E} = \{x^T \in \mathbb{R}^V \mid \text{$T$ is an elimination tree of $G$}\}$.

\begin{lemma} \label{lemma3:polymatroid}
The base polytope $\mathbf{B}(f)$ coincides with the convex hull of 
$\mathcal{E}$.
\end{lemma}
\begin{proof}
Let $T$ be an elimination tree of $G$, and 
let $\sigma$ be an ordering of the vertices in 
$V$ such that 
$T$ is compatible with $\sigma$.
Let $z$ be the extreme point of $\mathbf{B}(f)$ generated 
by $\sigma$. 
Then, Lemma~\ref{lemma2:polymatroid} implies that $z = x^T$. 

Let $z$ be an extreme point of $\mathbf{B}(f)$. 
Recall that there exists an ordering 
$\sigma = (v_1,v_2,\dots,v_n)$ of the vertices in $V$  
such that $z$ is generated by $\sigma$.
Define 
the elimination tree $T$ as that obtained by removing 
the vertices in $V$ according to the ordering 
$\sigma$. (Namely, we first remove $v_1$, then 
we remove $v_2$, and so on.)  
Then, $T$ is compatible with 
$\sigma$.
Thus, Lemma~\ref{lemma2:polymatroid} implies that $z = x^T$. 
\end{proof}

The following lemma immediately follows from 
Lemmas~\ref{lemma:D09} and \ref{lemma3:polymatroid}.

\begin{lemma} \label{main_lemma:polymatroid}
The base polytope $\mathbf{B}(f)$ is a realization of 
the $G$-associahedron.
\qed
\end{lemma}

Thus, we complete the reduction from 
\textsc{Combinatorial Shortest Path on Graph Associahedra}
to 
\textsc{Combinatorial Shortest Path on Polymatroids}. 
Therefore, Theorem~\ref{thm:hardness-polymatroid} follows from Theorem~\ref{thm:hardness}. 

\section{Conclusion}

We prove that the combinatorial shortest path computation is hard on graph associahedra and base polytopes of polymatroids.
This evaporates our hope for resolving an open problem to obtain a polynomial-time algorithm for finding a shortest flip sequence between two triangulations of convex polygons and the rotation distance between two binary trees by generalizing the setting to graph associahedra. However, that open problem is still open, and we should pursue another way of attacking it.

\bibliography{graphassociahedron}

\end{document}